\documentclass[11pt]{article}

\usepackage{graphicx}
\usepackage{epsfig}
\usepackage{epstopdf}
\usepackage{amssymb}
\usepackage{amsthm}

\newenvironment{keywords}{\flushleft {\bf Keywords.}}{ }

\def\P{{\cal P}}
\def\C{{\cal C}}
\def\U{{\cal U}}
\def\M{{\cal M}}
\def\I{{\cal I}}
\def\S{{\cal S}}
\def\R{{\cal R}}
\def\F{{\cal F}}

\def\H{{\cal H}}

\def\D{{\cal D}}
\def\L{{\cal L}}

\newcommand{\RR}{\mbox{I\hspace{-.08cm}R}}

\theoremstyle{plain}
\newtheorem{theorem}{Theorem}
\newtheorem{lemma}[theorem]{Lemma}
\newtheorem{proposition}[theorem]{Proposition}
\newtheorem{corollary}[theorem]{Corollary}

\theoremstyle{definition}
\newtheorem{definition}{Definition}
\newtheorem{example}{Example}

\theoremstyle{remark}
\newtheorem*{remark}{Remark}

\title{On the Actual Inefficiency of Efficient Negotiation Methods}

\author{\begin{tabular}[t]{c@{\extracolsep{1em}}c}
{\bf Luca Barzanti}  & {\bf Marcello Mastroleo}\\
Department of Mathematics&Department of Mathematics\\
for Economic and &for Economic and\\
Social Sciences&Social Sciences\\
University of Bologna & University of Bologna \\
\texttt{luca.barzanti@unibo.it} & \texttt{marcello.mastroleo@unibo.it}
\end{tabular}}

\begin{document}
\maketitle

\begin{abstract}
In this contribution we analyze the effect that mutual information has on the actual performance of efficient negotiation methods. Specifically, we start by proposing the theoretical notion of Abstract Negotiation Method (ANM) as a map from the negotiation domain in itself, for any utility profile of the parties. ANM can face both direct and iterative negotiations, since we show that ANM class is closed under the limit operation. 

The generality of ANM is proven by showing that it captures a large class of well known in literature negotiation methods like: the Nash's bargaining solution, all the methods derived by multi criteria decision theory and, in particular, the ones based on Lagrange multipliers, the Single Negotiating Text which was used in the Camp David Accords, the Improving Direction Method, and so on. 

Hence we show that if mutual information is assumed then any Pareto efficient ANM is manipulable by one single party or by a collusion of few of them. At this point, we concern about the efficiency of the resulting manipulation. Thus we find necessary and sufficient conditions those make manipulability equivalent to actual inefficiency, meaning that the manipulation implies a change of the efficient frontier so the Pareto efficient ANM converges to a different, hence actually inefficient, frontier. 

In particular we distinguish between strong and weak actual inefficiency. Where, the strong actual inefficiency is a drawback which is not possible to overcome of the ANMs, like the Pareto invariant one, so its negotiation result is invariant for any two profiles of utility which share the same Pareto frontier, we present. While the weak actual inefficiency is a drawback of any mathematical theorization on rational agents which constrain in a particular way their space of utility functions.

For the weak actual inefficiency we then state a principle of Result's Inconsistency by showing that 
to falsify theoretical hypotheses is rational for any agent which is informed about the preference of the other, even if the theoretical assumptions, which constrain the space of agents' utilities, are exact in the reality, {\it i.e.} the preferences of each single agent are well modeled. In essence we show that, under weak actual inefficiency assumption, any mathematical model which correctly capture the reality, it produces inconsistent results.
\end{abstract}

\begin{keywords}
Negotiation, Pareto Efficiency, Mutual Information, Manipulation, Collusion, Result's Inconsistency by Rationality.
\end{keywords}
\tableofcontents
\newpage

\begin{flushright}
\emph{\scriptsize{Engage people with what they expect; it is what they are able to discern and confirms their projections. It settles them into predictable patterns of response, occupying their minds while you wait for the extraordinary moment - that which they cannot anticipate. (Sun Tzu)}}
\end{flushright}
\vspace{2cm}
\section{Introduction}\label{Sec1}

The rise of efficiency by rationality is not trivial when the interaction between two or more Decision Makers (DM's) is considered. Clearly in finding a trade-off between parties preferences, the crucial role is always played by mutual information; for example, in \cite{Kersten2001}, the author shows how knowing mutual utilities allows to reach efficiency by means of the straightforward illustration of the ``orange''. 

Nevertheless, in a conflict situation the idea of sharing proper preferences with the opponents is impracticable, thus the rise of the necessity in having an extern trusted party which can find the best settlement by knowing everything of everyone.

For this reason in the last century, the problem of how to implement such extern trusted party has been widely approached under different perspectives giving rise to the Negotiation Analysis, an interdisciplinary effort which aims to draw a prescriptive theory for negotiators, see \cite{RaiffaCD} of Raiffa which is a pioneer of this field. 

Between all the theories of interest, the two most prolifically used are Game Theory and Multiple Criteria Decision Making, see \cite{Sebinius08} for a punctual survey of the last half century evolution of the Negotiation Analysis concept. 

In game theoretical approach to negotiation, see \cite{Brams:1990rp} for a wide range of negotiation techniques, there is not a real extern party but its function is instead achieved by the same negotiating parties which coordinates them selfs by means of a predetermined scheme of interaction called game. 

In the other approach, instead, the real extern party, which is called negotiator, privately collects informations related to individual preferences and then he implements a multi criteria decision making algorithm to find a ``win-win'' settlement which is Pareto efficient.

The choice of these two theory is natural since the first one models interaction dynamics while the second one trade-offs.

Moreover, due to the flourish of electronic transactions in the last decade, negotiation techniques have the new perspective to became automated services those can be worldwide offered to find Pareto efficient gains in several contexts see, for example: \cite{Teich94} for a classification of modeling aids, \cite{1174197} for a report on the development of interactive electronic negotiations in the supply chain, \cite{869133} for an artificial intelligence perspective of negotiating techniques, and \cite{Ehtamo200154} for a potentially universal negotiation platform.

Even if both game theoretical and multi criteria decision making approaches to negotiation achieve the same results, the techniques, which are based on the latter theory, are widely imposing as the best candidates to automated negotiations, since they open to the whole literature about optimization and control, due to their numerical and analytical root. Nevertheless there are tacit differences between the two problems and the most relevant is that in multi criteria decision making there is one single DM, who possesses multiple conflicting interests, while in negotiation each utility is associated to a different DM. This slight discrepancy is determinant in practices because a single rational DM has no incentive to be inefficient, while the same DM is pushed by rationality to pursue his own good even if this would imply a global inefficiency of the negotiation dynamic, whenever that global inefficiency leads him a payoff which is better than the one he would get in the efficient case. 

In fact rationality is a local property that is owned by each DM, which does all his best to pursuit his own good, whereas efficiency is a global property in the sense that it is referred to the outcome of their interaction. It is not difficult to see that fully rational DM's often fail to reach an efficient interaction if they negotiate without an external aid.

However the growth of computational power, that caused the born of electronic negotiations, is itself a potential danger for this field of applications; in fact it is not difficult to imagine a software specifically designed to exploit eventual lacks of a negotiation protocol. For example in \cite{RePEc:eee:ejores:v:194:y:2009:i:2:p:452-463}, the authors propose a volume based method to learn mutual information during the negotiation so to reach efficient settlements, which can be reversed engineered to be used, within any other negotiation method, by a party to learn mutual informations and to exploit them during the negotiation process. 

The rule of  mutual information is determinant and well studied, for example, in the Social Choice theory, which is close to negotiation, where the well-known result due to Arrow, in \cite{ArrowOriginal}, states that it is impossible to aggregate personal welfare functions efficiently and independently to irrelevant alternatives if the absence of a dictatorial party is assumed. Moreover Gibbard, in \cite{GibbardOriginal}, and Satterthwaite independently showed that by renouncing to the independence of irrelevant alternatives in seeking of non-dictatorship leads to manipulable welfares aggregations. Hence a party can obtain a better ending welfare by making-up his own declarations so to exploit the knowledge of the other social parties. 

Both results are originally formulated into a finite space of alternatives and other differences subsist between social choice and negotiation problems, in fact a negotiation can default and the {\it status quo} can be an inner point of the negotiation domain, for example in iterative negotiation methods or in the improving ones. However there are several evidences that Social Choice like statements could hold also within negotiation. In fact, it is well known that some of the most diffused Pareto efficient negotiation methods are sensible to manipulatory or collusive behaviors, like, for example, the Adjusted Winner procedure (\cite{Brams1995}), which is not truthful as detailed in \cite{RaiffaCD}, and the Improving Direction Method (\cite{740667}) which is proven to be both not truthful and information leaking in \cite{MastroleoAML}. 

Moreover in \cite{Mastroleo11}, the authors shows that if the {\it status quo} is exogenous to the negotiation domain, like in Social Choice theory, both Arrow's and Gibbard-Satterthwaite's results can be extended to negotiation over continuous issues, thus in this contribution we decided to explore what mutual information implies if the {\it status quo} is in the negotiation domain, like in iterative negotiation methods.

In the first part of Section \ref{Sec2} we fix the notions which we are going to use later in the section to define the class of Abstract Negotiation Methods. This general framework allows us to move the analysis of a negotiation methods from a prescriptive point of view to an analytical one, giving us the possibility to state general results about what can be really achieved in a negotiation process. In Subsection \ref{Sec2.1} we explore some notions of Pareto efficiency for ANMs and in Subsection \ref{Sec2.2} we show the equivalence between direct Pareto efficient, or ``one-shot'' Pareto efficient, and improving acyclic iterative negotiations by showing that each iterative ANM has a limit which is an ANM that achieve the same results in just one iteration. 

We show, in Section \ref{Sec3}, that the hypotheses, in particular the continuity ones, used to define ANMs are minimal since they are able to generalize a wide class of different approach to negotiations, namely: in Subsection \ref{Sec3.1} the Egalitarian and the Nash's Bargaining solution, as presented respectively in \cite{Kalai77} and \cite{Nash50}, the Multi-Criteria Decision Making strategies based on Lagrange Multipliers, like the ones used in \cite{Heiskanen99,Johanssona08}, in Subsection \ref{Sec3.2} 
and, in Subsection \ref{Sec3.4}, the Single Negotiating Text framework, introduced by Fisher in \cite{FisherCD}, and implemented in the Improved Direction Method which is developed in \cite{740667,Ehtamo200154}.

In Section \ref{Sec4} we investigate the connection between efficient negotiation methods and manipulability and we present a theorem which states that efficient $2$-party ANMs are always manipulable. The same theorem plays an important role in multi party negotiation when collusions are assumed. In Subsection \ref{Sec4.1} we give necessary and sufficient conditions under which a rational manipulation does not alter the real Pareto frontier of parties utilities and we show that both the bargaining solutions, discussed in Subsection \ref{Sec3.1}, satisfy this condition.

In Section \ref{Sec5} we show that there are methods which cannot be manipulated so to preserve real efficiency and in subsection \ref{Sec6} we study the conditions those make actual inefficient all the efficiently manipulable methods. We underline an issue due to constraining the utility's space, frequent practice in literature, which compromises the possibility to verify the truthfulness on the hypotheses which drive those constrains.

\section{Abstract Negotiation Methods}\label{Sec2}
In what follows, we assume that negotiation domain $\D\subsetneq \RR^m$ is closed, convex and bounded, where the dimension $m$ corresponds to the number of parameters those are going to be defined during the negotiation process. Moreover with $\D^\circ$ we indicate the interior of $\D$ and with $\partial\D$ its boundary.

For example in fund allocation problems, where a fixed maximum budget $B$ has to be split between $m$ possible alternatives, the negotiation domain is the set $\D = \{(x_1,\ldots,x_m)\in\RR^m|\sum{x_j} \leq B,\,x_j \geq 0\}$, which contains all the possible ways to split a quantity bounded by $B$, among the $m$ alternatives. 

We also assume that there are $n$ fully rational Decision Makers (DM's) which take part to the negotiation, where full rationality means that each party $i$ has both a well defined utility function $u_i:\D\rightarrow\RR$, {\it i.e.} super-level sets of $u_i$ are all closed, convex and they forms a decreasing succession with a unique global maximum $\hat{u}_i\in\D$, and the capability to pursue his own good. We indicate with $\U$ the set of all possible functions on $\D$ satisfying the aforementioned properties and a profile of utilities with $\vec{u}\in\U^n$.

With the next definitions we fix some notions those are going to be relevant in what follows.

\begin{definition}\label{Def:feasible}
Given a utility function $u\in\U$ and a point $p\in\D$, the regions of $\D$ those are \emph{improving}, \emph{strictly improving} and \emph{equivalent to} $p$ according to $u$, are $\F(u,p) = \{x\in\D|u(x) \geq u(p)\}$, $\F^+(u,p) = \{x\in\D|u(x) > u(p)\}$ and $\I(u,p) = \{x\in\D|u(x) = u(p)\}$, respectively.
\end{definition}

The same can be done for a profile of utility functions $\vec{u}\in\U^n$ by means of the intersection operator $\bigcap$ in the Borel algebra of $\D$ subsets.

\begin{definition}
Given a profile of utility functions $\vec{u}\in\U^n$ and a point $p\in\D$, the region of $\D$ that \emph{jointly} improves $p$ according to $\vec{u}$ is
$\F(\vec{u},p) = \bigcap\limits_{i=1}^{n} \F(u_i,p).$
In the same way, the regions of $\D$ those improves strictly and are equivalent to $p$ according to $\vec{u}$ are, respectively, $\F^+(\vec{u},p) = \bigcap\limits_{i=1}^{n} \F^+(u_i,p)$ and
$\I(\vec{u},p) = \bigcap\limits_{i=1}^{n} \I(u_i,p)$.
\end{definition}

In order to deal with negotiation axiomatically, we give the following general definition of negotiation method which does not concern about the actual procedure that allows the ending settlement to be reached. 

\begin{definition}[Abstract Negotiation Method]\label{Def:NegotiationMethod}
Given a negotiation domain $\D$, an $n$-party \emph{Abstract Negotiation Method} (ANM) is a map $\M:\U^n\times\D\rightarrow\U$.
\end{definition}

In this setting, the final settlement, or outcome, of a negotiation process, which implements $\M$ and starts by $x_0\in\D$, is naturally defined as the concave optimization problem
\begin{equation}\label{Eq:fins}
\S(\M,\vec{u},x_0)=\arg\max_{x\in\D} \M(\vec{u},x_0)(x).
\end{equation}
Clearly the maximization problem $\S$ is well defined since $\M(\vec{u},x_0)$ is in $\U$ and $\D$ is bounded, closed and convex. Furthermore the gap between the evaluation of $\M$ and the search of $\S$ is merely a numerical matter and several generalizations can be introduces within the space $\U$ without compromising the tractability of the global theory. For example, if the concavity is relaxed to quasi-concavity the problem can be still numerically solved by using relaxation techniques like the one introduced in \cite{Ortega70} or the concave equivalent of Graduated Non-Convexity approximation, see \cite{Blake87}.

This definition of negotiation model considers as equivalent all those actual negotiating procedures which achieve the same ending settlement whenever they are fed with the same profile of utilities and starting point as inputs. Thus allowing us to move the focus from a prescriptive point of view, which studies the way parties have to be coordinated by the negotiator, to an analytical one, which permits us to reason about what a negotiation process can actually achieve in terms of parties goods.

\begin{remark}
The presence of starting point $x_0$, which is endogenous to the negotiation domain $\D$, determines the possibility that the negotiation can default. In fact at least a party exists which rejects the negotiation outcome whenever $\S(\M,\vec{u},x_0)\not\in\F(\vec{u},x_0)$ otherwise $\S(\M,\vec{u},x_0)$ is improving everyone's utility so, by means of rationality, no party is going to reject the negotiation outcome. The condition under which a negotiation does never default, is the Pareto efficiency which is the topic of the next subsection.
\end{remark}

In order to avoid that the negotiation outcome is heavily affected either by small perturbations of the starting point or small changes of the utility profile, we introduce the following notion of continuity.

\begin{definition}\label{Def:UPC}
A negotiation method $\M$ on $\D$ is \emph{continuous} if and only if $\S(\M,\vec{u},\cdot)$ is a continuous map from $\D$ to itself, for any fixed $\vec{u}\in\U^n$; and a pseudometric $\mu$ on $\U^n$ exists for which $\S(\M,\cdot,x_0)$ is continuous form $\U^n$ to $\D$, for any fixed $x_0\in\D$.
\end{definition}

The choice of the pseudometric $\mu$ is sensible and its metric identification is directly connected to the invariance properties that usually arise in dealing with utility functions representing preference systems. In the following, we refers only to continuous ANM even if not directly stated.

An important property that a good negotiation method is likely to have is the \emph{fairness}; we refer to the weaker condition of \emph{symmetry} meaning that the ANM is invariant under any change of the parties order. 

\begin{definition}
An ANM $\M$ is \emph{symmetric} if and only if $$\S(\M,\vec{u},x) = \S(\M,\vec{u}_\sigma,x)$$ for every permutation $\sigma$ of the parties index.
\end{definition}


\subsection{Pareto Efficiency}\label{Sec2.1}

Let us enter in the main topic of this contribution by indicating with $\P(\vec{u})$ the set of all those points which cannot be further improved, formally
$$\P(\vec{u}) = \{x\in\D|\F(\vec{u},x) = \{x\}\},$$
and with $\H(C,x)$ the set of hyperplanes which support the convex set $C$ at its border point $x$. Then, the following proposition is a necessary and sufficient condition, in therms of supporting hyperplanes, for a point $x$ to lie on $\P(\vec{u})$, which we are going to use in the following sections.

\begin{proposition}
$x\in\P(\vec{u})$ if and only if $H \in \H(\F(\vec{u}_A,x)) \cap \H(\F(\vec{u}_{-A},x))$ exists which separate $\F(\vec{u}_A,x)$ and $\F(\vec{u}_{-A},x)$, for all $A\subset\{1,\ldots,n\}$ 
\end{proposition}

\begin{proof}
$\F(\vec{u}_A,x)$ and $\F(\vec{u}_{-A},x)$ are always convex intersecting at $x$; the condition of Pareto efficiency $\F(\vec{u}_A,x) \cap \F(\vec{u}_{-A},x) = \{x\}$ is then equivalent to their separability through an hyperplane.
\end{proof}

If the functions $u_i$'s are all concave then the Geoffrion's result, in \cite{Geoffrion68}, holds; thus the weighing method $<w,\vec{u}>$, when the weight vector $w$ varies over the set $W^+_n=\{\mathbf{w}\in\RR^n|w_i\geq 0 \forall i = 1,\ldots,n \mbox{ and } \sum w_i = 1\}$, characterizes all the points in $\P(\vec{u})$.

\begin{definition}
An ANM $\M$ is \emph{Pareto efficient} if and only if $\S(\M,\vec{u},x_0)\in\P(\vec{u})$, for all $(\vec{u},x_0)\in\U^n\times\D$.
\end{definition}

By definition a Pareto efficient settlement cannot be improved without penalizing at least a party, thus the importance of these points since they are terminal. In fact if a negotiation reaches a Pareto efficient point then there is at least a party which is going to reject a change of that settlement. In this sense Pareto efficient settlement can be considered fixed points under the action of the map $\M(\vec{u},\cdot):\D\rightarrow\D$.

Notice that Pareto efficient ANM are direct methods since they reach directly a fixed point.

\subsection{Iterative ANMs and Limit Efficiency}\label{Sec2.2}

In real negotiations, where implementability matters, there are several difficulties in dealing with direct Pareto efficient ANM, especially due to the necessity to express a closed form function representing a party's preferences. Hence direct methods often leave place to iterative ones, those allow to gradually reach the ending settlement by locally searching for joint gains, like in the single negotiating text (\cite{FisherCD}) or in the methods developed in \cite{Harri2001,Heiskanen99}. In spite of implementability, direct methods are widely used for theoretical proposes due to their tractability.

For these reasons in this Subsection we firstly introduce iterative ANMs and then we state necessary and sufficient conditions under which an iterative method has a direct equivalent which is Pareto efficient; thus we can still refer to direct methods, in place of iterative ones, to analyze what efficient negotiations can actually achieve.

\begin{definition}[Iterative ANM]\label{def:n-INP}
Given an ANM $\M:\U^n\times\D\rightarrow\D$, an {\it n-parties Iterated ANM} $\M^t$ (IANM) which implements $\M$ is a negotiation method that follows the sequent scheme for any given $x_0\in\D$:
\begin{enumerate}
\item $t\leftarrow 0$;
\item while $\S(\M,\vec{u},x_t)\not\in \{x_k\}_{k\leq t}$ and $u_i(\S(\M,\vec{u},x_t)) \geq u_i(x_t)$ for all $i\in\{1,\ldots,n\}$ do
\begin{enumerate}
\item[2.1.] $x_{t+1} \leftarrow \S(\M,\vec{u},x_t)$;
\item[2.2.] $t\leftarrow t+1$;
\end{enumerate}
\item $x^*\leftarrow x_t$
\item return $x^*$.
\end{enumerate}
\end{definition}

The first condition $\S(\M,\vec{u},x_t)\not\in \{x_k\}_{k\leq t}$, in line 2, ensures that the negotiation does not loop between two or more points that are considered equivalent by all parties thus avoiding the so called {\it bad faith} negotiation, {\it i.e.} when parties negotiate without a real improving intent. 

Whereas the second condition derives from the assumption of rationality which implies parties to continue the negotiation until at least one of them is actually penalized by the next contract.

Both this conditions can be relaxed by constraining the class of ANMs, thus requiring the negotiator to be some how smarter. 

\begin{definition}
An ANM $\M$ is {\it acyclic} if and only if, for all $\vec{u}\in\U^n$ and $x_0\in\D$ the condition $\S(\M,\vec{u},x_t)\not\in \{x_k\}_{k\leq t}$ is never satisfied for $x_t\not\in\P(\vec{u})$.
\end{definition}

\begin{definition}
An ANM $\M$ is {\it improving} if and only if, for all $\vec{u}\in\U^n$ the following inequality holds:
\begin{equation}\label{Eq:Improving}
\S(\M,\vec{u},x)\in\F(\vec{u},x),
\end{equation}
moreover $\M$ is {\it jointly improving} if and only if $\S(\M,\vec{u},x)\in\F^+(\vec{u},x)$, whenever $x\not\in\P(\vec{u})$.
\end{definition}

\begin{proposition}
If $\M$ is jointly improving then it is acyclic.
\end{proposition}

The following proposition allows us to relax the while loop conditions in case of at least improving ANMs.

\begin{proposition}
If $\M$ is at least improving and acyclic then the stop conditions are redundant.
\end{proposition}

Nevertheless we decide to do not suppress the utility not-penalizing condition because parties should always have the right to stop the negotiation process and moreover its redundancy vanish whenever parties declare false statements to the negotiator, hence its necessity is going to be evident from Section \ref{Sec4} when strategical behavior about declarations is analyzed. 

The next proposition shows that the last definitions are not redundant.

\begin{proposition}
Any direct Pareto efficient ANM $\M$ is improving and acyclic. Moreover if $\M$ is symmetric then it is jointly improving
\end{proposition}

\begin{proof}
The first part is trivial by looking at the definition of Pareto efficiency. Let us so consider a Pareto efficient ANM $\M$ which is not jointly improving then, for every $x_0\in\D$, at least an $i$ exists for which $u_i(\S(\M,\vec{u},x_0)) = u_i(x_0)$, meaning that an $(n-1)$-party $\tilde\M$ exists and $\M(\vec{u},x_0) = \tilde\M(\vec{u}_{-i},x_0)\cdot I(\F(\vec{u},x_0))$, with $I(\F(\vec{u},x_0))$ the characteristic function of $\F(\vec{u},x_0)$. Clearly $\M(\vec{u},x_0)$ is not symmetric.
\end{proof}

\begin{remark}
Being jointly improving does not implies symmetry and a simple counterexample is a weighted version of the Nash's bargaining solution $$\M(\vec{u},x_0) = (u_1(x)-u_1(x_0))^{\alpha_1}\cdot(u_2(x)-u_2(x_0))^{\alpha_2},$$ for any $\alpha_1,\,\alpha_2>0$ with $\alpha_1 + \alpha_2 = 1$ and $\alpha_i\neq1/2$.
\end{remark}

By indicating with $|\F(\vec{u},x_0)|$ the finite value of $\int \F(\vec{u},x_0) dx$, which is the measure of the feasible region of $\vec{u}$ at $x_0$, the following limit theorem partially reverses the previous proposition and it allows us to consider only Pareto efficient ANM, thus to remove the complication that iteration represents.

\begin{theorem}\label{Thm:limit}
$\M^t$ always converges to a Pareto efficient settlement if and only if $\M$ is at least improving and acyclic.
\end{theorem}

\begin{proof}
$(\Rightarrow)$ If $\M^t$ converges to a Pareto efficient settlement then the stop conditions at line 2 in Definition \ref{def:n-INP} never hold; thus $\S(\M,\vec{u},x)\in\F(\vec{u},x)$ and $\M^t$ does not cycle.  

\noindent $(\Leftarrow)$ Let $\M$ be an improving and acyclic ANM and $\M^t$ its iteration, then let us assume by absurd that for a fixed $(\vec{u},x_0)\in\U^n\times\D$ an $m>0$ exists such that $\lim_{t\rightarrow\infty} |\F(\vec{u},x_t)| > m$, thus $x^*\not\in\P(\vec{u})$ and $\S(\M,\vec{u},x^*) = x^*$, since $\S(\M,\vec{u},\cdot):\D\rightarrow\D$ is continuous and $\D$ bounded, closed and convex thus compact, which contradicts the hypothesis that $\M$ is acyclic.
\end{proof}

\section{Remarkable Examples}\label{Sec3}

In this Section we demonstrate the generality of continuous ANMs by showing that they fit some of the well known in the literature negotiating procedures. 

The first example is trivial and has the only propose to show that a dummy method which returns always the same point regardless its inputs is an ANM.

\begin{example}\label{Ex:Dummy}
Let us consider in a negotiation domain $\D$ a point $s^*\in\D$, then the {\it dummy} negotiation method is $$\M_{s^*}(\vec{u},x_0) = \sup\limits_{s\in\D}d(s,s^*)^2 - d(x,s^*)^2.$$
Clearly $\M_{s^*}$ is constant in both its parameters, thus it is continuous with respect to both utility profiles and starting points. Clearly $\M_{s^*}$ is not Pareto efficient and it is going to default whenever $s^*\not\in\F(\vec{u},x_0)$.
\end{example}

The next examples are less trivial and they represent, by looking at the different ways to approach negotiation problems, some of the most specialized methods.

\subsection{Bargaining Solutions}\label{Sec3.1}
Let $\D^k_2=\{x\in\RR^2|x_1,x_2\geq0,\,x_1+x_2\leq k\}$ represents the space of all possible splits among two parties of a certain finite good. 

We can consider the {\it egalitarian bargaining solution} (see \cite{Kalai77}) in $\D$
\begin{equation}\label{Eq:EgalitarianBargaining}
\texttt{e}(\vec{u},x_0) = \arg\max\limits_{x\in \D} \min(u_1(x)-u_1(x_0),u_2(x)-u_2(x_0)),
\end{equation}
where $x_0$ represents an already agreed split, the so called {\it status quo}, that both parties want to improve. This solution can be clearly achieved by implementing the ANM
\begin{equation}\label{Eq:EgalitarianBargainingANM}
\M_\texttt{e}(\vec{u},x_0) = \min(u_1(x)-u_1(x_0),u_2(x)-u_2(x_0))
\end{equation}
whose ending settlement according to Equation \ref{Eq:fins} is exactly the egalitarian bargaining solution. Moreover the continuity of $\M_\texttt{e}$ is trivial with respect to the starting point $x_0$ and the continuity with respect to the utility profile $\vec{u}$ can be proven by fixing a point $\bar x \in \D$ and using the translation invariant pseudometric on $\U$
$$\mu_\texttt{e}(u,v) = \sup\limits_{x\in\D} |u(x) - v(x) - (u(\bar x) - v(\bar x))|.$$

We can also consider the {\it Nash's bargaining solution} (see \cite{Nash50})
\begin{equation}\label{Eq:NashBargaining}
\texttt{n}(\vec{u},x_0) = \arg\max\limits_{x\in\D}(u_1(x)-u_1(x_0))\cdot(u_2(x)-u_2(x_0)),
\end{equation}
where, again, $x_0$ is again the {\it status quo}. $\texttt{n}(\vec{u},x_0)$ can be accomplished by using the following
\begin{equation}\label{Eq:NashBargainingANM}
\M_\texttt{n}(\vec{u},x_0) = (u_1(x)-u_1(x_0))\cdot(u_2(x)-u_2(x_0)),
\end{equation}
which is clearly continuous with respect to the {\it status quo} $x_0$ and, by fixing $\bar x,\,\bar{\bar x}\in\D$, the affine invariant pseudometric on $\U$
$$\mu_\texttt{n}(u,v) = \sup\limits_{x\in\D} |(v(\bar x)-v(\bar{\bar x}))\cdot (u(x) - u(\bar x)) - (u(\bar x)-u(\bar{\bar x}))\cdot (v(x) - v(\bar x))|$$
ensures utility profiles continuity of $\M_\texttt{n}$.

\subsection{Lagrange Multipliers}\label{Sec3.2}
In \cite{Heiskanen99}, Heiskanen provided an iterative method which uses Lagrange multipliers to reach Pareto efficient settlements in the negotiation over continuous issues with concave utility functions. Later, in \cite{Johanssona08}, the authors presented a negotiation algorithm to compute optimal consensus point in linear utility spaces which also relays on Lagrange multipliers.

In general the strategy behind all negotiation methods which are based on Lagrange multipliers, is to transform the original negotiation problem
\begin{equation}\label{Eq:LagrangeOri}
\left\{
\begin{array}{l l}
\arg\max\limits_{x\in\D} f(\vec{u_i})(x) &    \\
<x,\nabla u_i(x_0)> \geq 0 &  1\leq i\leq n
\end{array}
\right.
\end{equation}
where $f:\RR^m\rightarrow\R$ is an aggregating function, for example in \cite{Heiskanen99} $f(\vec{u}) = <w,\vec{u}>$ with $w \in W^+_n$, into a new one
\begin{equation}
\L: \,\arg\max\limits_{x\in\D} f(\vec{u})(x) + \sum\lambda_i\cdot <x,\nabla u_i(x_0)>
\end{equation}
which has a strictly concave objective and the solution of the problem is Pareto optimal, meaning that Karush--Kuhn–-Tucker conditions hold, for a suitable choice of $\lambda_i$'s.

Clearly $\M_\L(\vec{u},x_0) = f(\vec{u}) + \sum\lambda_i\cdot <x,\nabla u_i(x_0)>$ is an ANM whose solution is exactly like the above one and its continuity is trivial with respect to both $x_0$ and $\vec{u}$ choices.

Notice that both the bargaing solutions of the previous example can be expressed as a Lagrange multipliers Problem.
\subsection{Improving Direction Method}\label{Sec3.4}

In \cite{740667, Ehtamo200154, Harri2001}, the authors proposed a jointly improving negotiation method, aligned with the single negotiating text (see \cite{FisherCD,FisherUryCD,RaiffaCD}), that requires little more than local evaluations of $\nabla u_i$, to find an improving settlement. 

The map, that the {\it improving direction method} proposes, can be written as $\texttt{idm}(\vec{u},x) = l(\vec{u},x,g(\vec{u},x))$, where $g: \U^n\times\D\rightarrow B^m$ associates to any point of $\D$ the generalized bisector of the angle spanned by $F(u_i,x)$'s in $x$, according to the solution of the following product maximization problem, see \cite{Ehtamo200154} for details,
\begin{equation}\label{Eq:G3}
g(\vec{u},x)=
\left\{
\begin{array}{l l}
\max\limits_d \prod\limits_{i = 1}^n (\nabla u_i(x)/||\nabla u_i(x)||, d) &    \\
d_i(x_t)\cdot d \geq 0 &  1\leq i\leq n  \\
d \in F(\vec{u},x)  & \\
||d||^2 = 1     
\end{array}
\right.
\end{equation}
and $l:\U^n\times\D\times B^m\rightarrow \D$ evaluates the maximum step that can be done in $g(\vec{u},x)$ direction without penalizing any of the $P_i$'s, formally
\begin{equation}\label{Eq:L}
l(\vec{u},x,d) = x + \left(\min_i\left(\arg\max_{\lambda_i}u_i(x+\lambda_i\cdot d)\right)\right) \cdot d.
\end{equation}

In this case, differently from the ones above, the method has not a closed functional form since it is already an optimization method, hence we skip the definition of the ANM $\M_{\texttt{idm}}$ which implements IDM, since it can always be written as the dummy ANM centered in $\texttt{idm}(\vec{u},x)$, and we directly analyze the continuity of $\S(\M_{\texttt{idm}}, \vec{u}, x) = \texttt{idm}(\vec{u},x)$. 

$\S(\M_{\texttt{idm}}, \vec{u}, x)$ is clearly continuous as map of $\D$ in itself for any fixed $\vec{u}$ but proving that it is utility profile continuous need some calculation. Thus let us consider the preference invariant relation $p(u,v) \Leftrightarrow \leq_u = \leq_v$ and the preference invariant pseudometric 
\begin{equation}
\mu_{\texttt{idm}}(u,v) = \sup_{x\in\D,\atop {||\nabla u(x)||\neq 0\atop||\nabla v(x)||\neq 0}} 1 -  \left(\frac{(\nabla u(x),\nabla v(x))}{||\nabla u(x)||\cdot||\nabla v(x)||}\right)^2 
\end{equation}
which, in essence, controls the maximum orthogonal component between the normalized gradients of $u$ and $v$; $\mu_{\texttt{idm}}$ is well defined in the space of continuously differentiable utility function, which the authors consider, and it constrains the solution of both equation \ref{Eq:G3} and \ref{Eq:L}, thus ensuring that $\texttt{idm}$ is the ending settlement of a continuous ANM.   

\begin{remark}
IDM is an iterative method which is jointly improving and by means of Theorem \ref{Thm:limit} we know that it exists a direct Pareto efficient method which achieve the same result.
\end{remark}

\section{Pareto Efficiency Implies Manipulability}\label{Sec4}

In this Section we show that Pareto efficiency always implies manipulability whenever the {\it status quo} is endogenous to the negotiation domain $\D$. 

\begin{definition}
An ANM $\M$ is \emph{manipulable} if and only if for every $\vec{u}\in\U^n$ and $x_0$ at least an $i\in\{1,\ldots,n\}$ and an $\tilde{u}_i\in\U$ exist such that
$$u_i(\S(\M,\vec{u}_{-i},\tilde{u}_i,x_0)) > u_i(\S(\M,\vec{u},x_0)).$$ 
\end{definition}

\begin{proposition}
If $\U'\subset\U$ and $\M$ is manipulable within $\U'^n$ then it is manipulable also in $\U^n$.
\end{proposition}

This proposition is trivially true, but it allows us to restrict to 
$$\U' = \{u\in\U|\arg\max_{s\in\D}u_i(s)\in\partial\D\}.$$
We call $C_i(\M,\vec{u}_{-i},x_0)$ the set of possible manipulation outcomes of the $i$-th party within $\U'$, formally: 
$$C_i(\M,\vec{u}_{-i},x_0) =  \{x\in\D|\exists u_i\in\U' \mbox{ such that } x=\S(\M,u_i,\vec{u}_{i},x_0)\}.$$

\begin{lemma}
If $\M$ is efficient and not manipulable by the party $i$ then \\$C_i(\M,\vec{u}_{-i},x_0)$ is the frontier of a convex set, for all $(\vec{u}_{-i},x_0)\in\U^{(n-1)}\times\D$.
\end{lemma}

\begin{proof}
Firstly we prove that $C_i(\M,\vec{u}_{-i},x_0)$ has an empty interior. By assuming the contrary, let us take an inner point $\bar x \in C_i^\circ(\M,\vec{u}_{-i},x_0)$ and let us consider its associated utility $\bar u_i$, then 
$$\{\bar x\} \subsetneq \F(\bar u_i,\bar x) \cap C_i(\M,\vec{u}_{-i},x_0),$$
implying that there exists another utility which improves $\bar u_i$'s performance contradicting the hypotheses.

Let now $\bar x$ be again in $C_i(\M,\vec{u}_{-i},x_0)$ and $\bar u_i$ be its associated utility, then $\M$ is not manipulable by $i$ if and only if $<x-\bar x,\nabla \bar u_i>$ is a supporting hyperplane for $C_i(\M,\vec{u}_{-i},x_0)$ at $\bar x$ in $\D$. The convexity follows by varying the couple $(\bar x,\bar u_i)$.
\end{proof}

\begin{theorem}\label{Thm:2man}
Any 2-party efficient ANM $\M$ is manipulable by at least a party within $\U'\times\U'$.
\end{theorem}

\begin{proof}
We show that if $\M$ is not manipulable by party 2 then it is manipulable by party 1. Hence, for all $\vec{u}\in\U^2$ and $x_0\not\in\P(\vec{u})$ preferred by both parties to the opponent's best preference in $\D$, we have that $C_2(\M,u_1,x_0)$ is the frontier of a convex set and clearly $x_0\in C_2(\M,u_1,x_0)$.
Let us now assume that $\M$ is not manipulable also by party 1, thus let us consider $C_1(\M,u_2,x_0)$ which holds the same properties of $x_0\in C_2(\M,u_1,x_0)$. Moreover
$$\{x_0, \S(\M,\vec{u},x_0)\}\subset \left( C_1(\M,u_2,x_0) \cap C_2(\M,u_1,x_0)\right),$$ 
$$\hat{u}_2 = arg\max_{x\in\D}u_2(x)\in C_1(\M,u_2,x_0)$$ 
and
$$\hat{u}_1 = \arg\max_{x\in\D}u_1(x)\in C_2(\M,u_1,x_0).$$

Since party two cannot manipulate $\M$, $H(x) = <x-\S(\M,\vec{u},x_0),\nabla \bar u_2>$ is a supporting hyperplane of $C_2(\M,u_1,x_0)$ at $\S(\M,\vec{u},x_0)$ and, in order to be $\S(\M,\vec{u},x_0)$ Pareto efficient, $H$ should be supporting also to $C_1(\M,u_2,x_0)$ but this is impossible, since the two convex frontier meets in $x_0$ and $H(\hat{u}_2) < 0$.
\end{proof}

\begin{corollary}
If an ANM $\M$ is Pareto efficient and symmetric then it is manipulable by each party.
\end{corollary}

\begin{example}\label{Ex2}
The ANM $\M_{d}(\vec{u},x_0) = u_i\cdot I(\F(\vec{u},x_0))$ is not symmetric and it manipulable by all parties but the $i$-th, since they can reshape $\F(\vec{u},x_0)$. Furthermore $\M_{d}$ is the equivalent, with an endogenous {\it status quo}, of the dictatorial method which is the only efficient and not manipulable method whenever the {\it status quo} is exogenous. Hence, differently than social choice theory, efficiency implies manipulability with no exceptions, in negotiation with a endogenous {\it status quo}.
\end{example}

The following definition introduces a generalization of the concept of manipulation, which naturally arises in multi party negotiation when some of them form a coalition by sharing the same intent to overcome the others.

\begin{definition}
A $j$-party \emph{collusion} $\C$ is a subset of $\{1,\ldots,n\}$ of cardinality $j$ such that $u_{i_1} = u_{i_2}$ for all $i_1,\,i_2\in \C$.
\end{definition}

\begin{theorem}
If an $n$-party efficient ANM $\M$ is not manipulable by the $i$-th party then $\M$ is manipulable by a collusion of the other $n-1$ parties.
\end{theorem}

\begin{proof}
Whenever all the parties but the $i$-th one collude, then $\M$ reduces to a 2-party efficient negotiation ANM and the proof comes easily by applying the Theorem \ref{Thm:2man}.
\end{proof}

The same arguments apply to the following Theorems.

\begin{theorem}
If an $n$-party efficient ANM $\M$ is not manipulable by the a collusion of all the parties but the $i$-th one, then $\M$ is manipulable by the $i$-th party.
\end{theorem}

\begin{theorem}
If an $n$-party efficient ANM $\M$ is not manipulable by $j$-party collusion, then it is manipulable by the collusion of the other $n-j$.
\end{theorem}

\subsection{Efficient Manipulability}\label{Sec4.1}
At this point to ask whether a possible manipulation does affect the real efficiency of a method is a legitimate question.

Let us indicate with $C(\M)\subset\{1,\ldots,n\}$ the set of the parties which can manipulate the method $\M$.

\begin{definition}
An efficient ANM $\M$ is \emph{efficiently manipulable} if and only if 
\begin{equation}
\exists u^*_i \in \left\{\arg\sup\limits_{\tilde{u}_i\in\U} u_i(\S(\M,\tilde{u}_i,\vec{u}_{-i},x_0))\right\}
\end{equation}
such that $\P(\tilde{u}_i,\vec{u}_{-i}) = \P(\vec{u})$, for all $(\vec{u},x_0)\in\U^n\times\D$ and all $i\in C(\M)$.
\end{definition}

The following is a sufficient condition to ensure efficient manipulability which is suitable to capture the bargaining solutions described in Section \ref{Sec3}.

\begin{proposition}\label{Prop:EMeps}
If $\M$ is efficient and for every fixed $(\vec{u},x_0)\in\U^n\times\D$, $i\in C(\M)$ and $\epsilon > 0$ an $u_i^\epsilon\in\U$ exists such that $\P(u_i^\epsilon, \vec{u}_{-i}) = \P{\vec{u}}$ and 
$$|\S(\M,u_i^\epsilon,\vec{u}_{-i},x_0) - {\arg\max}_{x\in\F(\vec{u},x_0)} u_i(x) | < \epsilon$$
then $\M$ is efficiently manipulable.
\end{proposition}

\begin{proof}
If the hypotheses are true, then the $i$-th party, in looking for an manipulating utility, can gain $$\lim\limits_{\epsilon\rightarrow 0} \S(\M,u_i^\epsilon,\vec{u}_{-i},x_0) = \arg\max_{x\in\F(\vec{u},x_0)} u_i(x),$$
without moving the Pareto efficient frontier. The result follows since the right-hand side of the previous equality is the best choice within $\F(\vec{u},x_0)$ for party $i$ and the uniqueness of the $\max$. 
\end{proof}

\begin{corollary}\label{Cor:EgEM}
The egalitarian bargaining method $\M_{\texttt{e}}$ is efficiently manipulable.
\end{corollary}

\begin{proof}
Let $x^\epsilon\in\P(\vec{u})$ be a point which is far less than $\epsilon$ form $$x_1^* = \arg\max_{x\in\F(\vec{u},x_0)} u_1(x)$$ and let $a \in (0, u_2(x^\epsilon)/(u_1(x_1^*)-u_1(x_0)))$ then $u_1^\epsilon = a\cdot u_1$ satisfies the hypothesis of Proposition \ref{Prop:EMeps}. 
\end{proof}

\begin{corollary}
The Nash's bargaining solution $\M_{\texttt{n}}$ is efficiently manipulable.
\end{corollary}

\begin{proof}
It is always possible to choose an $\alpha_\epsilon > 1$ such that $u_1^\epsilon(x) = (u_1(x)-u_1(x_0))^{\alpha_\epsilon}$ meets the hypothesis of Proposition \ref{Prop:EMeps} for every $\epsilon > 0$.
\end{proof}

\begin{corollary}
The method $\M_{w}(\vec{u},x_0) = <w,\vec{u}-\vec{u}(x_0)> + \ln(I(\F(\vec{u},x_0)))$, with $w\in W^+_n$ is efficiently manipulable.
\end{corollary}

\begin{proof}
It is similar to the proof of Corollary \ref{Cor:EgEM}.
\end{proof}

\begin{example}\label{Ex3}
The ANM, in example \ref{Ex2}, $\M(\vec{u},x_0) = u_i\cdot I(\F(\vec{u},x_0))$ is not efficiently manipulable. It suffices to see it in 2-party with $\D$ planar ($m = 2$). In fact if $\M(u_1,u_2,x_0) = u_2\cdot I(\F(u_1,u_2,x_0))$ then to find the best manipulation of party 1 is equivalent to solve the following variational problem 
\begin{equation}
V = \left\{
\begin{array}{l}
\arg\sup\limits_{\I\in C^1([0,1]\rightarrow\D)} u_{j}\left(\arg\sup\limits_{x\in\I_{x_0}} u_i(x) \right) \\
\I(0) = x_0 \\
\I(1) \in\partial\D\\
\I(y) \geq \I(x) + <\nabla \I(x),(x-y)>\\
<\nabla\I(0),\nabla u_2(0)>\, > 0
\end{array}
\right.
\end{equation}
by finding the right shape of the possible indifference curves $\I$ of $x_0$, which has not necessarily a Pareto preserving solution, meaning that the preferred point of party 2 on $\I$ is not necessarily on $\P(u_1,u_2)$.
\end{example}

\begin{theorem}\label{Thm:EP}
If $\M$ is efficient then manipulatory behaviors do not affect its actual efficiency if and only if $\M$ is efficiently manipulable.
\end{theorem}

\begin{proof} $(\Leftarrow)$ Let $\vec{u}^*_{C(\M)}$ be the best manipulatory response of the parties in $C(\M) = \{i_1,\ldots,i_j\}$ then by the definition of efficient manipulability 
$$\P(u^*_{i_1},\vec{u}_{-i_1}) = \P(u^*_{i_1},u^*_{i_2},\vec{u}_{-\{i_1,i_2\}}) = \ldots = \P(\vec{u}^*_{C(\M)},\vec{u}_{-C(\M)}) = \P(\vec{u}).$$
$(\Rightarrow)$ If manipulatory behaviors do not affect the efficiency of the method, then let us fix $i \in  C(\M)$ and its best manipulation $u_i^*$ then, for every possible manipulation, not necessarily the best one, of the others $\vec{u}'_{C(\M)\setminus\{i\}}$, we have that $\S(\M,u_i^*,\vec{u}'_{C(\M)\setminus\{i\}},\vec{u}_{-C(\M)},x_0) \in \P(\vec{u})$ which implies that $\P(u_i^*,\vec{u}_{-i}) = \P(\vec{u})$, by the arbitrarily and $\vec{u}'_{C(\M)\setminus\{i\}}$ and the theorem by the arbitrarily of $i\in C(\M)$.
\end{proof}

\begin{remark}
If an ANM is efficiently manipulable, then there is no rational reason for a party to manipulate the negotiation by changing the Pareto frontier. In fact any change of the Pareto frontier would lead to a stronger loss when the other parties manipulate the negotiation too.
\end{remark}

If a negotiation method is efficiently manipulable, then its efficiency is preserved by means of Theorem \ref{Thm:EP}, thus we now focus on the possible situations those the efficiency of a method is actually compromised by manipulatory behaviors. 

Notice that we use the term \emph{actual inefficiency} to underline that the method, which is proven to be efficient, does not loose its property to end with a Pareto efficient settlement but the inefficiency is the result of a malicious change of the Pareto frontier whose points are inefficient with respect to the real one.

\section{Actual Inefficiency}\label{Sec5}

Let us firstly present the strongest notion of Actual Inefficiency and then an example of a method which satisfies it.

\begin{definition}
Any efficient ANM $\M$ is \emph{Strong Actual Inefficient} if and only if it is not efficiently manipulable over $\U$.
\end{definition}

Let us consider the 2-party budget allocation problem and for any point $x_0\in\D$ and any $u\in\U$ we define the \emph{strictly improving frontier} of $u$ at $x_0$ the set
\begin{equation}
B^u_{x_0} = \left\{x\in\D| x=x_0+\bar\lambda (y-x_0),\,\\ \bar\lambda=\arg\sup\limits_{\lambda\geq0}u(x_0+\lambda (y-x_0)),\, y\in\D \right\}.
\end{equation}

If $x\in B^u_{x_0}$ then the line $\overline{xx_0}$ supports $\F(u,x_0)$ at $x_0$, furthermore if $u_1,\,u_2\in \U$ and $x_0$ is strictly preferred by $u_1$ to the best outcome of $u_2$ and {\it vice versa}, then any $x^*\in B^{u_1}_{x_0}\cap B^{u_2}_{x_0}$, with $x^*\neq x_0$, lies on $\P(u_1,u_2)$. Moreover $x^*$ holding the aforementioned properties is unique whenever $u_1,\,u_2\in C^1(\D)$.

This simple argument gives us the possibility to present within this context a simple 2-party Pareto invariant ANM 
\begin{equation}
\M_\P(u_1,u_2,x_0) = 
\left\{
\begin{array}{l l}
B^{u_1}_{x_0}\cap B^{u_2}_{x_0} \setminus \{x_0\} & \mbox{if } u_1,\,u_2\in C^1(\D) \\
\arg\sup\limits_{x\in B^{u_1}_{x_0}\cap B^{u_2}_{x_0}} |x-x_0| & \mbox{otherwise.}
\end{array}
\right.
\end{equation}

$\M_\P$, despite its easy implementability, has a strong theoretical value since it is never efficiently manipulable. In fact it is not difficult to check that for any $\tilde{u}_1$, such that $\P(\tilde{u}_1,u_2) = \P(u_1,u_2)$, it results $\M_\P(\tilde{u}_1,u_2,x_0) = \M_\P(u_1,u_2,x_0)$.

Moreover $C_1(\M_\P,u_2,x_0) = B^{u_2}_{x_0}$, thus to maximize the manipulation outcome of party 1 is equivalent to find $x^* = \arg\sup\limits_{x\in B^{u_2}_{x_0}} u_1$.

\subsection{Weak Actual Inefficiency}\label{Sec6}

We now explore a weaker form of actual inefficiency which natural arises in all those contexts, like economics, where there is an active speculation about the properties of utility functions with consequent constrains of the utilities' space.

\begin{definition}
Any efficient ANM $\M$ is \emph{Weak Actual Inefficient} on $\U'$ if and only if it is efficiently manipulable over $\U\supsetneq\U'$ but not over $\U'$.
\end{definition}

The following theorem characterizes a prototypical space $\U'$ in which the weak actual inefficiency arises.

\begin{theorem}\label{ThmUpr}
If $\U'$ is the space of convex combinations of \\$U = (U^1,\ldots,U^l)\in\U^l$ and the following properties hold 
\begin{enumerate} 
\item $U^j(\arg\sup(U^i)) = \inf(U^j)$ and $\arg\sup(U^i)\in\partial\D$, $\forall i,j\in\{1,\ldots,l\}$, $i\neq j$,
\item $\exists m:\D\rightarrow W^+_l$ continuous and bijective such that \\$x = \arg\sup_{y\in\D} m(x)\cdot U(y)$,
\end{enumerate}
then any improving 2-party ANM $\M$ is weakly actual inefficient on $\U'$, whenever there is perfect competition among parties.
\end{theorem}

\begin{proof}
By the perfect competition assumption, we have that $u_1,\,u_2$ have maxima $\hat{u}_1,\,\hat{u}_2\in\partial \D$, $u_1(\hat{u}_2) = \inf(u_1)$ and {\it vice versa}. Moreover, since $m$ is one-to-one and continuous, $C_1(\M,u_2,x_0) \cap C_2(\M,u_1,x_0) \cap \P(u_1,u_2) = \{\S(\M,u_1,u_2,x_0)\}$, like in the proof of Theorem \ref{Thm:2man}. Thus any manipulation compromises the efficiency of $\M$.
\end{proof}

\begin{example}
Let us consider again the 2-party budget splitting problem and, like in \cite{Harri2001,Ehtamo200154,740667,FRutilities}, let us think at the space $\U'$ consisting in all the convex combinations of the following utility functions $U^1 = \ln(x_1)$, $U^2 = \ln(x_2)$ and $U^3 = \ln(B-x_1-x_2)$. Clearly $\U'$ satisfies the hypotheses of Theorem \ref{ThmUpr}, thus any efficient method, included the ones presented in the aforementioned articles, are weakly actual inefficient on $\U'$. 
\end{example}

Despite strong actual inefficiency cannot be overcame since it is a property that globally holds on $\U$, weak actual inefficiency is local on $\U'\subset\U$ and can be eliminated by enlarging $\U'$. 

If $\U'$ induces weak actual inefficiency of an efficiently manipulable $\M$, then there are two different scenarios
\begin{enumerate}
\item parties are forced to stick with utilities in $\U'$, like for example in all those software implementations where party can only express the parameters $a\in W^+_l$ which codifies their utilities;
\item parties can ``escape'' from $\U'$.
\end{enumerate}

If the first scenario holds then the weak actual inefficiency becomes substantially a strong actual inefficiency. Otherwise in the latter scenario parties are going to falsify the theoretical hypotheses which led to identify the class $\U'$ by picking a manipulating utility outside $\U'$ even if their real one is in $\U'$, meaning that the theoretical hypotheses are true.

Let us call the triplet $M = (\D,\U',\M)$ a model of interaction in $\D$, with preferences in $\U'$ and according to the method $\M$, then the next theorem codifies the latter scenario.

\begin{theorem}
Let $M = (\D,\U',\M)$ be a model of the interaction within an $n$-party system $S$. If $\U\supsetneq \U'$ exists, such that $\M$ is weak actual inefficient on $\U'$, then the theoretical results obtained by applying $M$ are going to be inconsistent with the actual dynamic of $S$, whenever there is a manipulatory behavior.
\end{theorem}

Moreover, if interactions between the parties of a systems are done according to an efficient manipulable ANM $\M$, then nothing more then the obvious can be said about the dynamic of the systems. In fact, it is possible to say no more that it is going to converge on the Pareto efficient fortier, which is trivial by the efficiency of $\M$. In addition there is no way to identify the real method that the system uses to gain efficiency since it is impossible to discriminate between any two ones which are efficiently manipulable.

\section{Conclusions}
We introduced the notion of abstract negotiation method which, in essence, captures some of the direct and recursive negotiation methods which are well-known in literature. We showed that if the {\it status quo} is endogenous to the negotiation domain then efficiency always implies manipulability by either parties or collusion of them.

Then we focused on the effects those manipulatory behaviors have on the real efficiency of efficient negotiation methods. In exploring the topic, we derive necessary and sufficient conditions for a manipulation to do not alter the real Pareto efficient frontier; then we focus in looking for all those situations under which this is not true. So we show that there are methods, like the presented Pareto invariant one, in which it is not possible to manipulate efficiently a negotiations, and also that there are subspaces of utility functions in which manipulations always compromise the real efficiency. Regarding the latter situation we underline that it is possible to overcome the actual inefficiency by allowing parties to exit from the aforementioned space of utility.

There are several possibility to extend the results presented in this contribution, which we are going to investigate. One over all: to overcome the impossibility to negotiate efficiently and in a not manipulable way by renouncing to the determinism of the negotiation method.

\bibliographystyle{alpha}
\bibliography{MPN}
\end{document}